\documentclass[12pt]{article}
\usepackage{a4}
\usepackage{amsmath}
\usepackage{amssymb}
\usepackage{amsfonts}
\usepackage{amsthm}
\usepackage{graphicx}
  \newtheorem{thm}{Theorem}[section]
 \newtheorem{cor}[thm]{Corollary}
 \newtheorem{prop}[thm]{Proposition}
 \newtheorem{defn}[thm]{Definition}
 
 \newtheorem{lemma}[thm]{Lemma}
 \newtheorem{rem}[thm]{Remark}

\renewcommand{\span}{span}

\renewcommand{\a}{\alpha}
\renewcommand{\b}{\beta}
\renewcommand{\d}{\delta}
\newcommand{\la}{\lambda}

\newcommand{\C}{\mathbb{C}}
\newcommand{\R}{\mathbb{R}}
\renewcommand{\S}{\mathbb{S}}
\newcommand{\N}{\mathbb{N}}
\newcommand{\T}{\mathbb{T}}

\newcommand{\D}{\mathbb{D}}
\newcommand{\Z}{\mathbb{Z}}

 \title{Orthogonal expansions related to compact Gelfand pairs}
 \author{Christian Berg, Ana P. Peron and Emilio Porcu}
 \date{\today}

\begin{document}

 \maketitle
 
 \begin{abstract} Given a compact Gelfand pair $(G,K)$ and a locally compact group $L$, we characterize the class $\mathcal P_K^\sharp(G,L)$ of continuous positive definite functions $f:G\times L\to \C$ which are bi-invariant in the $G$-variable with respect to $K$. The functions of this class are the functions having a uniformly convergent expansion $\sum_{\varphi\in Z} B(\varphi)(u)\varphi(x)$
for $x\in G,u\in L$, where the sum is over the space $Z$ of positive definite spherical functions $\varphi:G\to\C$ for the Gelfand pair, and $(B(\varphi))_{\varphi\in Z}$ is a family of continuous positive definite functions on $L$ such that $\sum_{\varphi\in Z}B(\varphi)(e_L)<\infty$. Here $e_L$ is the neutral element of the group $L$. For a compact abelian group $G$ considered as a Gelfand pair $(G,K)$ with trivial $K=\{e_G\}$, we obtain a characterization of $\mathcal P(G\times L)$ in terms of Fourier expansions on the dual group $\widehat{G}$.

The result is described in detail for the case of the Gelfand pairs $(O(d+1),O(d))$ and $(U(q),U(q-1))$ as well as for the product of these Gelfand pairs. 

The result generalizes recent theorems of Berg-Porcu (2016) and Guella-Menegatto (2016).  
 \end{abstract}



\section{Introduction} 

In \cite{B:P} Berg and Porcu found an extension of Schoenberg's Theorem from \cite{S} about positive definite functions on spheres in euclidean spaces.
Schoenberg's Theorem has played an important role in statistics because covariance kernels for isotropic random fields on spheres are modelled via the class $\mathcal P(\S^d)$ of continuous
functions $f:[-1,1]\to\R$, which are positive definite in the sense introduced by Schoenberg: For $n\in\N, \xi_1,\ldots,\xi_n\in \S^d$, the matrix $[f(\xi_j\cdot\xi_k)_{j,k=1}^n]$ is positive semidefinite. Here $\S^d$ denotes the $d$-dimensional unit sphere in euclidean space $\R^{d+1}$. 

The extension of Schoenberg's theorems  in \cite{B:P} was motivated by the need in statistics to consider data, which depend both on the position on the earth and on the time, i.e. random fields defined on the product $\S^d\times \R$. Assuming isotropy for the position and stationarity in time, one is lead to consider the class $\mathcal P(\S^d,\R)$ of continuous functions $f:[-1,1]\times \R\to \C$ such that the kernel $((\xi,u),(\eta,v))\mapsto f(\xi\cdot\eta,v-u)$ is positive definite on the space $\S^d\times\R$.

We recall that for an arbitrary non-empty set $X$, a kernel on $X$ is a function $k:X^2\to\C$, and it is called positive definite if for any $n\in\N$, any finite collection of points $x_1,\ldots,x_n\in X$ and numbers $c_1,\ldots,c_n \in \C$  one has
$$
\sum_{j,k=1}^n k(x_j,x_k)c_j\overline{c_k}\ge 0,
$$
i.e., the matrix $[k(x_j,x_k)_{j,k=1}^n]$ is hermitian and positive semidefinite. By $\mathcal P(X^2)$ we denote the class of positive definite kernels on $X$. The theory of positive definite kernels is treated in \cite{B:C:R}.

 It turns out that in the characterization of the class $\mathcal P(\S^d,\R)$, the additive group $\R$ can be replaced by an arbitrary locally compact group $L$. We recall that a continuous function $\varphi:L\to\C$ is called positive definite on $L$, in symbols $\varphi\in\mathcal P(L)$, if the kernel $(x,y)\mapsto f(x^{-1}y)$ is positive definite on $L$.  

The neutral element in the group $L$ is denoted $e_L$.

The following theorem holds:

\begin{thm}\label{thm:BP1}{\rm (Theorem 3.3 in \cite{B:P})} Let $d\in\N$  and let $f:[-1,1]\times L\to \mathbb C$ be a continuous function. Then $f$ belongs to $\mathcal P(\S^d,L)$ in the sense that the kernel 
$$
(\xi,u),(\eta,v))\mapsto f(\xi\cdot\eta,u^{-1}v)
$$
is positive definite on $\S^d\times L$,
 if and only if there exists
 a sequence of functions $(\varphi_{n,d})_{n\ge 0}$ from $\mathcal P(L) $ with $\sum \varphi_{n,d}(e_L)<\infty$ 
such that
\begin{equation}\label{eq:expand}
f(x,u)=\sum_{n=0}^\infty \varphi_{n,d}(u) c_n(d,x),\quad x\in[-1,1],\;u\in L.
\end{equation}
The above expansion is uniformly convergent for $(x,u)\in [-1,1]\times L$, and we have
\begin{equation}\label{eq:coef}
\varphi_{n,d}(u)=\frac{N_n(d)\sigma_{d-1}}{\sigma_d}\int_{-1}^1 f(x,u)c_n(d,x)(1-x^2)^{d/2-1}\,{\rm d}x.
\end{equation}
\end{thm}

Here we have used the notation
\begin{equation}\label{eq:Geg}
c_n(d,x)=C_n^{(\la)}(x)/C_n^{(\la)}(1),\quad \la=(d-1)/2
\end{equation}
for the ultraspherical polynomials $c_n(d,x)$ as normalized Gegenbauer polynomials $C_n^{(\la)}(x)$ for the parameter $\la=(d-1)/2$, cf. \cite{A:A:R}. The constants $\sigma_d$ and $N_n(d)$ are defined in Section 5.

Schoenberg's Theorem for $\mathcal P(\S^d)$ is the special case of the previous theorem, where the group $L=\{e_L\}$ is trivial. The functions in $\mathcal P(L)$ are then just non-negative constants.

In \cite{M:P} Menegatto and Peron proved a theorem for the complex unit sphere $\Omega_{2q}$ in $\C^q$, see Section 6, analogous to Schoenberg's Theorem for $\mathcal P(\S^d)$. In an attempt to extend their result by taking product with an arbitrary locally compact group $L$ similar to the extension of Schoenberg's Theorem in \cite{B:P}, we realized that the real and complex spheres are homogeneous spaces  associated with certain groups and subgroups of respectively orthogonal and unitary matrices, and these matrix groups form compact Gelfand pairs. Therefore, Schoenberg's Theorem for real spheres $\S^d$ and Menegatto-Peron's result for complex spheres can be viewed as special cases of the Bochner-Godement Theorem for Gelfand pairs, see \cite{God}. It turns out to be possible to extend the Bochner-Godement theorem for compact Gelfand pairs by taking product with an arbitrary locally compact group $L$, and it is the main purpose of this paper to give a proof of that.

We refer to Section 2 for the necessary background about
Gelfand pairs, but we shall briefly introduce the setting of our main result, Theorem~\ref{thm:M} below.  Let $G$ and $K$ be compact groups with $K$  a subgroup of $G$ such that $(G,K)$ is a compact Gelfand pair. The dual space $Z$ of positive definite spherical functions $\varphi:G\to\C$ is an orthogonal basis for the Hilbert space $L^2_K(G)^\sharp$ of square integrable functions on $G$,  which are bi-invariant  with respect to $K$. 

To a compact Gelfand pair $(G,K)$ and an arbitrary locally compact group $L$ we shall 
characterize the set   $\mathcal P_K^\sharp(G,L)$ of continuous positive definite functions $f:G\times L\to \C$, which are bi-invariant in the $G$-variable with respect to $K$. The normalized Haar measure on $G$ is denoted $\omega_G$.    

\begin{thm}\label{thm:M} Let $(G,K)$ denote a compact Gelfand pair, let $L$ be a locally compact group and let  $f:G\times L\to \mathbb C$ be a continuous function. Then
$f$ belongs to $\mathcal P_K^\sharp(G,L)$ if and only if there exists
 a function $B:Z\to \mathcal P(L)$ satisfying  $\sum_{\varphi\in Z} B(\varphi)(e_L)<\infty$ 
such that
\begin{equation}\label{eq:expand}
f(x,u)=\sum_{\varphi\in Z} B(\varphi)(u)\varphi(x),\quad x\in G, \;u\in L.
\end{equation}
The above expansion is uniformly convergent for $(x,u)\in G\times L$, and we have
\begin{equation}\label{eq:coef}
B(\varphi)(u)=\d(\varphi)\int_G f(x,u)\overline{\varphi(x)}\,{\rm d}\omega_G(x)
 \quad u\in L.
\end{equation}
\end{thm}

This theorem is stated again as Theorem~\ref{thm:main} in Section 3. The dimension
$\delta(\varphi)$ is defined in \eqref{eq:hfi}.

In Section 2 we outline the general theory of Gelfand pairs and specializes the theory to compact Gelfand pairs.

In Section 3 we formulate our results about the class  $\mathcal P_K^\sharp(G,L)$ and prove the results in Section 4. A compact abelian group $G$ can be considered as a compact Gelfand pair with $K=\{e_G\}$. Applying Theorem~\ref{thm:main} to this Gelfand pair leads to an expansion result for functions in $\mathcal P(G \times L)$, see Theorem~\ref{thm:main2} and Corollary~\ref{thm:Fou}.   

In Section 5 and 6 we treat the special cases of the real and complex spheres.

In Section 7 we use that
the product of compact Gelfand pairs is again a compact Gelfand pair. This makes it possible to obtain a recent result about positive definite functions on products of real spheres: Theorem 2.9  in 
\cite{G:M:P}. We deduce an analogous theorem for products of complex spheres.

In \cite{G:M} Guella and Menegatto have extended the result of \cite{B:P} about expansions
$$
\sum_{n=0}^\infty \varphi_{n,d}(u)c_n(d,x),\quad x\in [-1,1], u\in L,
$$
where the coefficient functions $\varphi_{n,d}$ belong to $\mathcal P(L)$, to expansions where
the coefficient functions are positive definite kernels $k_n(u,v)$ on an arbitrary set $X$.

We show in Section 8 that such a generalization is also possible in the framework of compact Gelfand pairs.

\section{Harmonic analysis on Gelfand pairs}

We start with some expository material related to positive definite functions and Gelfand pairs.
The main references are \cite{D},\cite{sasvari},\cite{vD},\cite{V:K},\cite{W}.

Let $G$ denote a locally compact group with neutral element $e_G$. The set $\mathcal P(G)$ of continuous positive definite functions on $G$, introduced in Section 1, is important in representation theory, harmonic analysis and especially in probability theory when $G=\R^n$. These functions are treated in many books,  see e.g. \cite[p. 255]{D} and \cite[p.14]{sasvari}.  It is known that any $\varphi \in \mathcal P(G)$ satisfies $\varphi(x^{-1})=\overline{\varphi(x)}$ and $|\varphi(x)|\le \varphi(e_G)$ for $x\in G$.

For a compact subgroup $K$ of $G$ we call a function $\varphi:G\to\C$ bi-invariant with respect to $K$ if
\begin{equation}\label{eq:bi}
\varphi(kxl)=\varphi(x),\quad x\in G, k,l\in K.
\end{equation}    

For a set $A$ of functions on $G$ we denote by $A_K^\sharp$ the set of  functions from $A$ which are bi-invariant with respect to $K$. In particular $C_K^\sharp(G)_c$ denotes the set of continuous complex-valued functions  on $G$
with compact support and bi-invariant with respect to $K$. It is easy to see that for $f,g\in C_K^\sharp(G)_c$ the convolution
\begin{equation}\label{eq:conv}
f*g(x)=\int_G f(y)g(y^{-1}x)\,d\omega_G(y),\quad x\in G,
\end{equation}
where $\omega_G$ denotes a left Haar measure on $G$, is again a bi-invariant function on $G$, and 
$C_K^\sharp(G)_c$ becomes a subalgebra of the group algebra $L^1(G)$.
  
We say that $(G,K)$ is a {\it Gelfand pair} if $C_K^\sharp(G)_c$ is  commutative, cf. \cite[p. 75]{vD} or
\cite[Part 3]{W}. The latter contains many equivalent conditions for $(G,K)$ to be a Gelfand pair.

For a Gelfand pair the group $G$ is necessarily unimodular, cf. \cite{B},\cite[p. 75]{vD}. 
A {\it spherical function} for $(G,K)$ is a continuous function $\varphi:G\to\C$ satisfying
\begin{equation}\label{eq:sph}
\int_K \varphi(xky)\,d\omega_K(k)=\varphi(x)\varphi(y),\quad x,y\in G;\quad \varphi(e_G)=1,
\end{equation}
where $\omega_K$ is Haar measure on $K$ normalized to $\omega_K(K)=1$.
A spherical function is necessarily bi-invariant with respect to $K$. In fact, from \eqref{eq:sph} with $y=e_G$ we get
$$
\int_K \varphi(xk)\,d\omega_K(k)=\varphi(x),\quad x\in G,
$$
 and the right invariance under $K$ follows. Putting $x=e_G$ in \eqref{eq:sph}, we similarly get
the left invariance under $K$.

The compact subgroup $K$ determines an equivalence relation $\sim$ in $G$ defined by $x\sim y$ if and only if $x=kyl$ for some $k,l\in K$. The equivalence classes are the compact sets $KxK,x\in G$, which are called double cosets. The set of double cosets is denoted $K\backslash  
G/K$ and it is a locally compact space in the quotient topology, cf. \cite{Mu}. Functions on $K\backslash G/K$  can be identified with functions on $G$ which are bi-invariant with respect to $K$.

The {\it dual space} of a Gelfand pair $(G,K)$ is the set $Z$ of positive definite spherical functions.
It is a locally compact space in the topology inherited from $C(G)$, which  carries the topology of uniform convergence on compact subsets of $G$, cf. \cite[p.83]{vD}. Let $M_b(Z)$ denote the set of positive finite Radon measures on $Z$.

The Fourier transform of a function $f\in L^1_K(G)^\sharp$ is the function $\widehat{f}:Z\to \C$
defined by
\begin{equation}\label{eq:Fou}
\widehat{f}(\varphi)=\int_G f(x)\overline{\varphi(x)}\,d\omega_G(x),\quad \varphi\in Z.
\end{equation}

 It is a continuous function on $Z$ vanishing at infinity, in symbols $\widehat{f}\in C_0(Z)$.

The set $\mathcal P_K^\sharp(G)$ of continuous positive definite bi-invariant functions on $G$ is characterized by a Bochner type theorem due to Godement \cite{God}.  

\begin{thm} (Bochner-Godement)\label{thm:BG} For any $\mu\in M_b(Z)$ the function
\begin{equation}\label{eq:BG}
f(x)=\int_Z \varphi(x)\,d\mu(\varphi),\quad x\in G
\end{equation}
belongs to $\mathcal P_K^\sharp(G)$, and any such function has the form \eqref{eq:BG} for a uniquely determined $\mu\in M_b(Z)$.
\end{thm}

The {\it Plancherel measure} $\nu$ on $Z$ is a positive Radon measure depending  on the normalization of $\omega_G$ such that the following holds:

\begin{thm}(Plancherel-Godement)\label{thm:PG} 
For any $f\in C_K^\sharp(G)_c$
\begin{equation}\label{eq:God-P}
\int_G|f(x)|^2\,d\omega_G(x)=\int_Z |\widehat{f}(\varphi)|^2\,d\nu(\varphi).
\end{equation}
The mapping $f\mapsto \widehat{f}$ extends uniquely to an isometric isomorphism of $L_K^2(G)^\sharp$ onto $L^2(Z,\nu)$. 
\end{thm} 

We finally need the following {\it inversion theorem}, cf. \cite[p. 84]{vD}.

\begin{thm} \label{thm:IN} For $f\in \mathcal P_K^\sharp(G)\cap L^1(G)$ we have $\widehat{f}\in L^1(Z,\nu)$
and the unique measure $\mu\in M_b(Z)$ such that \eqref{eq:BG} holds has density $\widehat{f}$
with respect to the Plancherel measure $\nu$, i.e.,
\begin{equation}\label{eq:IT}
f(x)=\int_Z \varphi(x) \widehat{f}(\varphi)\,d\nu(\varphi),\quad x\in G.
\end{equation}
\end{thm}

We now specialize to compact Gelfand pairs $(G,K)$, meaning that $G$ is assumed to be compact.
We shall always normalize the Haar mesure of $G$ so that $\omega_G(G)=1$.
The following result is  well-known in the context of the Peter-Weyl Theorem, but for the convenience of the reader we give a direct proof.

\begin{prop}\label{thm:orth} Let $(G,K)$ denote a compact Gelfand pair. Two different functions
$\varphi,\psi\in Z$ are orthogonal on $G$, i.e., 
\begin{equation}\label{eq:o} 
\int_G \varphi(x)\overline{\psi(x)}\,d\omega_G(x)=0.
\end{equation}
\end{prop}

 \begin{proof} Defining
$$
a=\int_G \varphi(x)\overline{\psi(x)}\,d\omega_G(x),
$$
we get for $y\in G$
\begin{eqnarray*} 
&& a\varphi(y)=\int_G \varphi(x)\varphi(y)\overline{\psi(x)}\,d\omega_G(x)
=\int_G\int_K\varphi(xky)\,d\omega_K(k)\overline{\psi(x)}\,d\omega_G(x)\\
&=&\int_K\int_G\varphi(xky)\overline{\psi(x)}\,d\omega_G(x) d\omega_K(k)
=\int_K\int_G\varphi(xy)\overline{\psi(xk^{-1})}\,d\omega_G(x)d\omega_K(k)\\
&=& \int_G \varphi(xy)\overline{\psi(x)}d\omega_G(x)=\int_G \psi(x)\varphi(x^{-1}y)d\omega_G(x)=\psi *\varphi(y),
\end{eqnarray*}
where we first inserted \eqref{eq:sph} for $\varphi\in Z$ and changed the order of integration. For the fourth equality sign we replaced $x$ by  $xk^{-1}$  and used invariance of integration. Next we used that $\psi$ is right invariant under $K$ so the inner integral is independent of $k$. Finally,  we used that $\overline{\psi(x)}=\psi(x^{-1})$ and replaced $x$ by $x^{-1}$ in the integration.

Similarly we get
\begin{eqnarray*}
&& a\psi(y)=\int_G \varphi(x)\psi(x^{-1})\psi(y)d\omega_G(x)=
\int_G\varphi(x)\int_K\psi(x^{-1}ky)d\omega_K(k)d\omega_G(x)\\
&=&\int_K\int_G \varphi(kx)\psi(x^{-1}y)d\omega_G(x)d\omega_K(k)=\varphi*\psi(y).
\end{eqnarray*}
Using that convolution is commutative we get $a(\varphi(y)-\psi(y))=0$ for all $y\in G$, but since $\varphi\neq \psi$, we get $a=0$. 
\end{proof}

It follows from Proposition~\ref{thm:orth} that the dual space $Z$  of a compact Gelfand pair 
is  discrete, and if $G$ is metrizable then $Z$  is a countable set.

\begin{rem}{\rm Any spherical function for a compact Gelfand pair is automatically  positive definite,
so the dual space $Z$ is the set of all spherical functions. For a proof see \cite[p. 86]{vD}.
}
\end{rem}

The compact homogeneous space $G/K$ of left cosets $\xi=xK, x\in G$ carries a unique  probability measure $d\xi$, which is invariant under the $G$-action $G \times G/K\to G/K$ given by 
\begin{equation}\label{eq:G/K}
(g,\xi)\to g\xi=(gx)K,\;\mbox{if}\; \xi=xK,\quad g,x\in G.
\end{equation}
 Functions $F:G/K\to \C$ are in one-to-one correspondence with functions $f:G\to\C$ which are right invariant functions under $K$, i.e., $f(g)=f(gk)$ for all $g\in G, k\in K$. The correspondence is given by $F(gK)=f(g), g\in G$ and the following integral formula holds
 \begin{equation}\label{eq:G/K1} 
\int_G f(x)\,d\omega_G(x)=\int_{G/K} F(g\xi)\,d\xi=\int_{G/K} F(\xi)\,d\xi, \quad g\in G.
\end{equation}

For each $\varphi\in Z$ and $g\in G$ the function $x\mapsto \varphi(g^{-1}x)$ on $G$ is right invariant under $K$, so it can be considered as a function
$\varphi_g\in C(G/K)$. It is a classical fact that 
\begin{equation}\label{eq:hfi}
H_\varphi:=\span \{\varphi_g \mid g\in G\}
\end{equation}
is a finite dimensional subspace of $C(G/K)$. The dimension of $H_\varphi$ is denoted $\d(\varphi)$.

The spaces $H_\varphi,\varphi\in Z$,  are mutually orthogonal in  $L^2(G/K,d\xi)$ which they span.  

It is known that the Plancherel measure is given by $\nu(\{\varphi\})=\d(\varphi)$ for $\varphi\in Z$, and the spherical functions from $Z$ form an orthogonal system, cf. Proposition~\ref{thm:orth}.

More precisely we have:
\begin{equation}\label{eq:OS} 
\int_G \varphi(x)\overline{\psi(x)}\,d\omega_G(x)=\left\{\begin{array}{ll}
0  & \;\mbox{if}\; \varphi\neq \psi,\\
1/\d(\varphi) & \;\mbox{if}\; \varphi=\psi.
\end{array}
\right.
\end{equation}

The previous three theorems take the following form, which gives  the expansion of functions
in $L_K^2(G)^\sharp$ after the orthonormal basis $\sqrt{\d(\varphi)}\varphi, \varphi\in Z$: 

\begin{thm}\label{thm:compBG} Let $(G,K)$ be a compact Gelfand pair.

(i) Each $f\in L_K^2(G)^\sharp$  has the orthogonal expansion
$$
f(x)\sim \sum_{\varphi\in Z} \d(\varphi)\widehat{f}(\varphi)\varphi(x),
$$
which converges in $L^2(G)$, and Parseval's formula holds:
$$
\int_G|f(x)|^2\,d\omega_G(x)=\sum_{\varphi\in Z}|<f,\sqrt{\d(\varphi)}\varphi>|^2=
\sum_{\varphi\in Z} \d(\varphi)|\widehat{f}(\varphi)|^2.
$$

(ii) A function $f\in C_K^\sharp(G)$ belongs to $\mathcal P_K^\sharp(G)$ if and only if there exist a family $(B(\varphi))_{\varphi\in Z}$ of non-negative numbers satisfying $\sum_{\varphi\in Z} B(\varphi)<\infty$ such that
\begin{equation}\label{eq:cgp}
f(x)=\sum_{\varphi\in Z} B(\varphi)\varphi(x),\quad x\in G,
\end{equation}
and $B(\varphi)=\d(\varphi)\widehat{f}(\varphi)$. The series in \eqref{eq:cgp} is uniformly convergent on $G$.
\end{thm}

\section{Main Results}

We start  with a formal definition of the main class to be discussed.

Let $(G,K)$ denote a compact Gelfand pair, and let $L$ denote an arbitrary locally compact group with neutral element $e_L$. We shall consider a subset of  the set $\mathcal P(G\times L)$ of continuous  positive definite functions  on the locally compact group $G\times L$.

\begin{defn}  The set of continuous positive definite functions $f:G\times L\to\C$ which are bi-invariant with respect to $K$ in the first variable is denoted $\mathcal P_K^\sharp(G,L)$.
\end{defn}

The following Proposition states some properties of $\mathcal P_K^\sharp(G,L)$ which are easily obtained. The proofs are left to the reader.

\begin{prop}\label{thm:elem}
\begin{enumerate}
\item[{\rm (i)}] For $f_1,f_2\in\mathcal P_K^\sharp(G,L)$ and $r\ge 0$ we have $rf_1, f_1+f_2$, $f_1 f_2\in\mathcal P_K^\sharp(G,L)$. 
\item[{\rm (ii)}] For a net of functions $(f_i)_{i\in I}$ from $\mathcal P_K^\sharp(G,L)$ converging pointwise to a continuous function $f:G\times L\to \C$, we have $f\in\mathcal P_K^\sharp(G,L)$.
\item[{\rm (iii)}] For $f\in\mathcal P_K^\sharp(G,L)$ we have $f(\cdot,e_L)\in\mathcal P_K^\sharp(G)$ and $f(e_G,\cdot)\in\mathcal P(L)$.
\item[{\rm (iv)}] For $f\in\mathcal P_K^\sharp(G)$ and $g\in\mathcal P(L)$ we have $f\otimes g\in\mathcal P_K^\sharp(G,L)$, where
$f\otimes g(x,u):=f(x)g(u)$ for $(x,u)\in G\times L$. In particular we have
$f\otimes 1_L\in \mathcal P_K^\sharp(G,L)$ and $f\mapsto f\otimes 1_L$ is an embedding of $\mathcal P_K^\sharp(G)$ into $\mathcal P_K^\sharp(G,L)$.
\end{enumerate}
\end{prop} 

Our first main theorem can be stated as follows. 
Note that a function $B:Z\to \mathcal P(L)$ can be considered as a family $(B(\varphi))_{\varphi\in Z}$ in $\mathcal P(L)$. In the applications $Z$ is a countable set, so we can think of
 $(B(\varphi))_{\varphi\in Z}$ as a sequence from
$\mathcal P(L)$. 

\begin{thm}\label{thm:main} Let $(G,K)$ denote a compact Gelfand pair, let $L$ be a locally compact group and let  $f:G\times L\to \mathbb C$ be a continuous function. Then
$f$ belongs to $\mathcal P_K^\sharp(G,L)$ if and only if there exists
 a function $B:Z\to \mathcal P(L)$ satisfying  $\sum_{\varphi\in Z} B(\varphi)(e_L)<\infty$ 
such that
\begin{equation}\label{eq:expand}
f(x,u)=\sum_{\varphi\in Z} B(\varphi)(u)\varphi(x),\quad x\in G, \;u\in L.
\end{equation}
The above expansion is uniformly convergent for $(x,u)\in G\times L$, and we have
\begin{equation}\label{eq:coef}
B(\varphi)(u)=\d(\varphi)\int_G f(x,u)\overline{\varphi(x)}\,{\rm d}\omega_G(x)=\d(\varphi)
\widehat{f(\cdot,u)}(\varphi),
 \quad u\in L.
\end{equation}
\end{thm}
When $L$ denotes the group consisting just of the neutral element, Theorem~\ref{thm:main} reduces to Bochner-Godement's Theorem. 

We give the proof in the next section, but notice that the infinite series in Equation \eqref{eq:expand} is uniformly convergent by Weierstrass' M-test since
$|B(\varphi)(u)|\le B(\varphi)(e_L)$ for $u\in L$ and $|\varphi(x)|\le 1$  for $x\in G$.
It is also clear that any function given by \eqref{eq:expand} belongs to $\mathcal P_K^\sharp(G,L)$ because of Proposition~\ref{thm:elem} and the fact that by definition
 $\varphi\in \mathcal P_K^\sharp(G)$ for any $\varphi\in Z$. The main point in Theorem~\ref{thm:main} is that all functions in $\mathcal P_K^\sharp(G,L)$ have an expansion 
\eqref{eq:expand}.

\medskip
If $G$ is a compact abelian group and $K=\{e_G\}$ is trivial, then $(G,K)$ is a compact Gelfand pair. A spherical function $\varphi$ is the same as a continuous homomorphism 
$$
\varphi:G\to \T:=\{z\in\C\mid |z|=1\},
$$
i.e., a continuous group character on $G$. This means that the dual space $Z$ of the Gelfand pair is equal to the dual group $\widehat{G}$. For the space defined in Equation \eqref{eq:hfi} we have $H_\varphi=\C \varphi$ for $\varphi\in\widehat{G}$,  hence  $\d(\varphi)=1$.

Specializing Theorem~\ref{thm:main} to this Gelfand pair leads to the following result: 

\begin{thm}\label{thm:main2} Let $G$ denote a compact abelian group with dual group $\widehat{G}$, let $L$ be a locally compact group and let  $f:G\times L\to \mathbb C$ be a continuous function. Then
$f$ belongs to $\mathcal P(G\times L)$ if and only if there exists
 a function $B:\widehat{G}\to \mathcal P(L)$ satisfying  $\sum_{\varphi\in \widehat{G}} B(\varphi)(e_L)<\infty$ 
such that
\begin{equation}\label{eq:expandab}
f(x,u)=\sum_{\varphi\in \widehat{G}} B(\varphi)(u)\varphi(x),\quad x\in G, \;u\in L.
\end{equation}
The above expansion is uniformly convergent for $(x,u)\in G\times L$, and we have
\begin{equation}\label{eq:coefab}
B(\varphi)(u)=\int_G f(x,u)\overline{\varphi(x)}\,{\rm d}\omega_G(x)=
\widehat{f(\cdot,u)}(\varphi),
 \quad u\in L.
\end{equation}
\end{thm}

For the special case $G=\T^N$ of the $N$-dimensional torus where $\widehat{G}=\Z^N$, we get the following result about Fourier series in $N$ variables:

\begin{cor}\label{thm:Fou} Let $L$ be a locally compact group and let $f:\R^N\times L\to\C$ be a continuous function periodic with period $2\pi$ in the $N$ real variables. Then $f\in\mathcal P(\R^N\times L)$ if and only if there exists a  multisequence $(\varphi_\mathbf{n})_{\mathbf{n}\in\Z^N}$ from $\mathcal P(L)$ 
 satisfying  
$$
\sum_{\mathbf{n}\in \Z^N} \varphi_{\mathbf{n}}(e_L)<\infty
$$ 
such that
\begin{equation}\label{eq:expandFou}
f(\mathbf{x},u)=\sum_{\mathbf{n}\in\Z^N} \varphi_\mathbf{n}(u) \exp(i\,\mathbf{n}\cdot\mathbf{x}),\quad \mathbf{x}\in \R^N, \;u\in L.
\end{equation}
The above expansion is uniformly convergent for $(\mathbf{x},u)\in \R^N\times L$, and we have
\begin{equation}\label{eq:coefFou}
\varphi_{\mathbf{n}}(u)=(2\pi)^{-N}\int_{[0,2\pi]^N} f(\mathbf{x},u) \exp(-i\,\mathbf{n}\cdot\mathbf{x})\,{\rm d}\mathbf{x},
 \quad u\in L.
\end{equation}
\end{cor}

\section{Proofs} 

\begin{lemma}\label{thm:top} Let  $H$ denote a locally compact group and let  $C\subset H$ be a non-empty compact set. For any open neighbourhood $U$  of $e_H\in H$ there exists a partition
of $C$ in finitely many non-empty disjoint Borel sets, say $M_j,j=1,\ldots,r$, with the property:
$$
\mbox{For\;} x,y\in M_j \mbox{\;one has\;\;} x^{-1}y\in U.
$$ 
\end{lemma}

\begin{proof} Given an open neighbourhood $U$ of $e_H\in H$, there exists a smaller open neighbourhood $V$ of $e_H$ such that $V^{-1}V\subset U$. By the definition of compactness, see \cite[p. 164]{Mu}, there exists a finite covering of  the compact set $C$ by left translates $x_jV$. 
Defining  $B_1=x_1V\cap C$ and
$$
B_j=x_jV\cap C \setminus \cup_{k=1}^{j-1} x_kV,\; j\ge 2,
$$
the non-empty sets among the $B_j$'s will form a finite partition $M_1,\ldots,M_r$ of $C$ such that each $M_j$ is contained in a left translate of $V$, and therefore we have for any $x,y\in M_j$
$$
x^{-1}y\in V^{-1}V\subset U.
$$ 
\end{proof}

The following Lemma is well-known, see \cite[p.256]{D}, but for the convenience of the reader we give a self-contained proof.
\begin{lemma}\label{thm:equi} Let $H$ be a locally compact group. For a continuous function $f: H\to \C$ the following are equivalent:
\begin{enumerate}
\item[{\rm(i)}] $f \in \mathcal P(H)$.

\item[{\rm(ii)}] $f$ is bounded and for any complex Radon measure $\mu$ on $H$ of compact support we have
\begin{equation}\label{eq:mu}
\int_{H}\int_{H} f(x^{-1}y)\,{\rm d}\mu(x)\,
{\rm d}\overline{\mu(y)}\ge 0.
\end{equation} 
\end{enumerate}
\end{lemma}

\begin{proof} "(i)$\implies $(ii)." Suppose first that (i) holds. As noticed, any positive definite function is bounded. For any discrete complex Radon measure of the form
$$
\sigma=\sum_{j=1}^n \alpha_j\delta_{x_j},
$$
where $x_1,\ldots x_n\in H$, $\alpha_1,\ldots, \alpha_n\in\C$, we have
\begin{eqnarray*}
\lefteqn{\int_{H}\int_{H} f(x^{-1}y)\,{\rm d}\sigma(x)\,{\rm d}\overline{\sigma(y)}} \\
&=& \sum_{k,l=1}^n f(x_k^{-1}x_l)\alpha_k\overline{\alpha_{l}}\ge 0.
\end{eqnarray*}

Let now $\mu$ denote an arbitrary complex Radon measure on $H$ with compact support $C$,
and let us consider the number
$$
I:=\int_{H}\int_{H} f(x^{-1}y)\,{\rm d}\mu(x)\,
{\rm d}\overline{\mu(y)},
$$
which is clearly real. We shall prove that $I\ge 0$, by showing that for  any $\varepsilon >0$ there
exists $J\ge0$ such that $|I-J|<\varepsilon||\mu||^2$, where $||\mu||$ is the total variation of the complex measure $\mu$, cf. \cite{R1}.

First of all
 $f(x^{-1}y)$ is uniformly continuous on the compact set $C\times C$. Thus, for given $\varepsilon>0$ there exists an
open neighbourhood $U$ of $e_H$ such that for all pairs
 $(x,y),(\tilde x,\tilde y)\in C\times C$ satisfying
$x^{-1}\tilde x\in U,\; y^{-1}\tilde y\in U$
we have
$$
 |f(x^{-1}y)-f(\tilde x^{-1}\tilde y)|<\varepsilon.
$$
Corresponding to $U$ we choose  a partition $M_j,j=1,\ldots,r$ of $C$  with the property of Lemma~\ref{thm:top}.
In particular
\begin{eqnarray*}
|f(x^{-1}y)-f(\tilde x^{- 1}\tilde y)|<\varepsilon,
\end{eqnarray*}
if 
$$
(x,y),(\tilde x,\tilde y)\in M_k\times M_l.
$$

Next
$$
I=\sum_{k,l=1}^r\int_{M_k}\int_{M_l} f(x^{-1}y)\,{\rm d}\mu(x)\,
{\rm d}\overline{\mu(y)},
$$
and if we choose an arbitrary point $x_j\in M_j$ and define $\alpha_j=\mu(M_j)$, $j=1,\ldots,r$, then 
$$
J:=\sum_{k,l=1}^r f(x_k^{-1}x_l)\alpha_k\overline{\alpha_l}\ge 0.
$$
Furthermore,
$$
I-J=\sum_{k,l=1}^r\int_{M_k}\int_{M_l} [f(x^{-1}y)
-f(x_k^{-1}x_l)]
\,{\rm d}\mu(x)\,
{\rm d}\overline{\mu(y)},
$$
and  by the uniform continuity
$$
|I-J|\le \varepsilon\sum_{k,l=1}^r |\mu|(M_k)|\mu|(M_l)= \varepsilon ||\mu||^2,
$$
where $|\mu|$ denotes the total variation measure. 

\medskip  

The implication  "(ii)$\implies$(i)" is easy by specializing the  measure $\mu$ to a complex discrete
measure concentrated in finitely many points. 
\end{proof}

{\it Proof of Theorem~\ref{thm:main}.} Suppose that $f$ belongs to $\mathcal P_K^\sharp(G,L)$ and let us consider the product measure $\mu:=\omega_G\otimes \sigma$ on $G\times L$,
where $\sigma$ is an arbitrary complex Radon measure on $L$ of compact support. By Lemma~\ref{thm:equi} for the locally compact group $H=G\times L$ applied to $\mu$ we get
\begin{equation}\label{eq:help1}
\int_G\int_G\int_L \int_L f(x^{-1}y,u^{-1}v)\,{\rm d}\omega_{G}(x)\,{\rm d}\omega_{G}(y)\,{\rm d}\sigma(u)\,
{\rm d}\overline{\sigma(v)}\ge 0.
\end{equation}
The  integral with respect to $x,y$ can be simplified to
$$
\int_{G}\int_{G} f(x^{-1}y,u^{-1}v)\,{\rm d}\omega_{G}(x)\,{\rm d}\omega_{G}(y)=\int_{G}
f(y,u^{-1}v)\,{\rm d}\omega_G(y)
$$
by invariance of $\omega_G$.
Therefore Equation \eqref{eq:help1}
amounts  to
\begin{equation}\label{eq:help2}
\int_{G}\int_L \int_L f(y,u^{-1}v)\,{\rm d}\omega_G(y)\,{\rm d}\sigma(u)\,
{\rm d}\overline{\sigma(v)}\ge 0.
\end{equation}
 
For arbitrary $\varphi\in Z$ we next apply Equation \eqref{eq:help2} to the function $f(x,u)\overline{\varphi(x)}$, which belongs to $\mathcal P_K^\sharp(G,L)$
by Proposition~\ref{thm:elem}.
This gives
\begin{equation}\label{eq:help3}
\int_{G}\int_L \int_L f(y,u^{-1}v)\overline{\varphi(y)}
\,{\rm d}\omega_G(y)\,{\rm d}\sigma(u)\,
{\rm d}\overline{\sigma(v)}\ge 0.
\end{equation}
The function $B(\varphi):L\to\mathbb C$ defined by
\begin{equation}\label{eq:help4}
B(\varphi)(u):= \delta(\varphi)\int_{G} f(y,u)\overline{\varphi(y)}\,{\rm d}\omega_G(y)=
\delta(\varphi)\widehat{f(\cdot,u)}(\varphi)
\end{equation}
is clearly continuous and bounded, and it is positive definite on $L$, because Equation \eqref{eq:help3} holds for all complex Radon measures $\sigma$ on $L$ with compact support. 

For each $u\in L$ the function $f(\cdot,u)$ belongs to $C_K^\sharp (G)\subset L^2_K(G)^\sharp$ and has  the orthogonal expansion in spherical functions
\begin{equation}\label{eq:L}
f(y,u)=\sum_{\varphi\in Z} B(\varphi)(u)\varphi(y), 
\end{equation}
where the convergence is in $L^2(G)$ with respect to $y\in G$ by Theorem~\ref{thm:compBG}.

When $u=e_L$, then $f(\cdot,e_L)\in\mathcal P_K^\sharp(G)$, and by the Bochner-Godement Theorem we have
$$
f(y,e_L)=\sum_{\varphi\in Z} B(\varphi)(e_L)\varphi(y),\quad y\in G. 
$$
Since $B(\varphi)(e_L)\ge 0$ with $\sum_{\varphi\in Z} B(\varphi)(e_L)<\infty$ and
$$
|B(\varphi)(u)\varphi(y)|\le B(\varphi)(e_L),
$$
the M-test of Weierstrass shows that the series on the right-hand side of \eqref{eq:L} 
converges uniformly to a continuous function $\tilde f(y,u)$ on $G\times L$.
In particular, for each $u\in L$ the series \eqref{eq:L} converges uniformly in $y\in G$ to $\tilde f(y,u)$, but this implies  convergence in $L^2(G)$ and therefore $f(y,u)=\tilde f(y,u)$ for almost all $y\in G$. Since these functions are continuous, we have equality  for all $y\in G$ .
\hfill $\quad\square$   

\medskip

\section{The sphere $\S^d$ in $\R^{d+1}$}

For $n\ge 1$ let $O(n)$  denote the group of orthogonal $n\times n$
real matrices $A$ and let $SO(n)$ denote the normal subgroup of those matrices $A\in  O(n)$ with determinant 1. In the case $n=1$ these  groups are the simple multiplicative groups $\{\pm1\}$ and $\{1\}$.  

The groups $O(d+1)$ and $SO(d+1)$ both operate  on the  real unit sphere of dimension $d$
$$
\S^{d}=\{x\in\R^{d+1} \mid ||x||^2=\sum_{k=1}^{d+1} x_k^2=1\},
$$
but while $O(d+1)$ operates transitively, this is the case for $SO(d+1)$ only for $d\ge 1$.

The surface measure of the sphere $\S^d$ is denoted $\omega_d$ and it is of total mass
$$
\sigma_d=\omega_d(\S^d)=\frac{2\pi^{(d+1)/2}}{\Gamma((d+1)/2)}.
$$

We use the notation $e_1,\ldots,e_{d+1}$ for the standard basis in $\R^{d+1}$.
The fixed-point group of the matrices $A\in O(d+1)$  satisfying $Ae_{1}=e_{1}$,
is of the form
$$
A=\begin{pmatrix}
1 & 0 \\
0 & \tilde A
\end{pmatrix}, 
$$
where $\tilde A\in O(d)$, the zero in the upper right corner represents a zero row vector of length $d$, and the zero in the lower left corner represents a zero column vector of length $d$.

Let $G$ denote the compact group $O(d+1)$. This shows that the fixed-point group $K$  of $e_{1}$ is isomorphic to $O(d)$ and in the following identified with $O(d)$. The mapping $A\mapsto Ae_1$ of
$G=O(d+1)$ onto $\S^d$ is constant on the left cosets $\xi=AK$, and hence induces a bijection of $G/K$ onto $\S^d$, and it is a homeomorphism. 
The pair $(G,K)$ is known to be a compact Gelfand pair, cf. \cite{vD},\cite{W}.

The mapping $A\mapsto Ae_1\cdot e_1$ of $G$ onto $[-1,1]$ is constant on the double cosets and if $Ae_1\cdot e_1=Be_1\cdot e_1$ for $A,B\in G$, then they belong to the same double coset. Therefore the space of double cosets $K\backslash G/K$ is homeomorphic to $[-1,1]$. This shows that complex functions on $G$ which are bi-invariant with respect to $K$, can be identified with functions  $f:[-1,1]\to\C$. In fact, for such a function, $A\mapsto f(Ae_{1}\cdot e_{1})$ is a bi-invariant function on $G$ and all bi-invariant functions on $G$ have this form. The bi-invariant functions depend only on the upper left corner $a_{11}$ of $A\in O(d+1)$.

The image measure of Haar measure $\omega_G$ on $G=O(d+1)$ under the mapping $A\mapsto Ae_1$ of $G$ onto $\S^d$ is the normalized surface measure $\omega_d/\sigma_d$. The image measure of $\omega_d/\sigma_d$ under the mapping $\xi\mapsto \xi\cdot e_1$ of $\S^d$ onto $[-1,1]$ is the probability measure on $[-1,1]$ with density 
$$
(\sigma_{d-1}/\sigma_d)\left(1-x^2\right)^{d/2-1}
$$
with respect to Lebesgue measure, cf. \cite{M}.

The spherical functions are precisely the normalized ultraspherical polynomials $c_n(d,x)$ given by \eqref{eq:Geg}. As bi-invariant functions on $O(d+1)$ they are  positive definite by Schoenberg's Theorem in \cite{S}.   

Therefore the dual space $Z$ of the compact Gelfand pair $(O(d+1),O(d))$ can be identified with $\N_0$ and we have
$$  
\d(c_n(d,x))=N_n(d)=\frac{(d)_{n-1}}{n!}(2n+d-1),\; n\ge 1,\quad \d(c_0(d,x))=N_0(d)=1.
$$
For $f\in L^1_K(G)^\sharp$ considered as a function on $[-1,1]$ we have
$$
\widehat{f}(n)=(\sigma_{d-1}/\sigma_d)\int_{-1}^1 f(x)c_n(d,x)(1-x^2)^{d/2-1}\,dx,\quad n\in\N_0.
$$
The space $\mathcal P_K^\sharp(G)$ can be identified with the space $\mathcal P(\S^d)$ from
\cite{B:P}, and for an arbitrary locally compact group $L$ the space $\mathcal P_K^\sharp(G,L)$
can be identified with $\mathcal P(\S^d,L)$. Theorem 3.3 of \cite{B:P} is a  special case of 
Theorem~\ref{thm:main}, and it can be  formulated  as follows:

\begin{thm}\label{thm:realsphere} Let $d\in\N$, let  $L$ be a locally compact group  and let $f:[-1,1]\times L\to \mathbb C$ be a continuous function. Then
$A\mapsto f(Ae_1\cdot e_1,u)$ belongs to $\mathcal P^\sharp_{O(d)}(O(d+1),L)$ if and only if there exists
 a sequence of functions $(\varphi_{n,d})_{n\ge 0}$ from $\mathcal P(L) $ with $\sum \varphi_{n,d}(e_L)<\infty$ 
such that
\begin{equation}\label{eq:expandr}
f(x,u)=\sum_{n=0}^\infty \varphi_{n,d}(u) c_n(d,x),\quad x\in[-1,1],\;u\in L.
\end{equation}
The above expansion is uniformly convergent for $(x,u)\in [-1,1]\times L$, and we have
\begin{equation}\label{eq:coefrs}
\varphi_{n,d}(u)=N_n(d)(\sigma_{d-1}/\sigma_d) \int_{-1}^1 f(x,u)c_n(d,x)(1-x^2)^{d/2-1}\,{\rm d}x.
\end{equation}
\end{thm}

\begin{rem}{\rm There is an apparent discrepancy between the symbol $\mathcal P(\S^d,L)$, which denotes a set of functions on $[-1,1]\times L$, and  $\mathcal P_K^\sharp(G,L)$, which denotes a set of functions on $G\times L$, but since these functions are bi-invariant with respect to $K$ in the $G$-variable, these functions can be considered as functions on $(K\backslash G/K)\times L$.  The double coset space $K\backslash G/K$ is homeomorphic to $[-1,1]$ in case of the Gelfand pair 
$(O(d+1),O(d))$, and therefore we think that the notation is reasonable.
}
\end{rem}

In all the above one could also have considered the compact Gelfand pair $(SO(d+1),SO(d))$ when
$d\ge 2$. For $d=1$ there is no need for Gelfand pairs since $SO(1)$ is trivial and  $SO(2)$ is commutative and can be identified with $\S^1$.

\section{The complex sphere $\Omega_{2q}$ in $\C^q$}

In the following we use the notation from \cite{M:P}.

For $q\ge 1$ let $U(q)$  denote the group of unitary $q\times q$
complex matrices $A$ and let $SU(q)$ denote the normal subgroup of those matrices $A\in  U(q)$ with determinant 1. In the case $q=1$ these  groups are the  multiplicative groups $\T=\{z\in\C \mid |z|=1\}$ and $\{1\}$. Since these are commutative we will always assume $q\ge 2$.  

The groups $U(q)$ and $SU(q)$ both operate  on the complex unit sphere of (real) dimension $2q-1$
$$
\Omega_{2q}=\{z\in\C^{q} \mid ||z||^2=\sum_{k=1}^{q} |z_k|^2=1\},
$$
but while $U(q)$ operates transitively, this is the case for $SU(q)$ only for $q\ge 2$.

Note that  $\Omega_{2q}$ is equal to $\S^{2q-1}$ if $\C^q$ is identified with $\R^{2q}$.

We use the notation $e_1,\ldots,e_{q}$ for the standard basis in $\C^{q}$.
The fixed-point group of the matrices $A\in U(q)$  satisfying $Ae_{1}=e_{1}$,
is of the form
$$
A=\begin{pmatrix}
1 & 0 \\
0 & \tilde A
\end{pmatrix}, 
$$
where $\tilde A\in U(q-1)$, the zero in the upper right corner represents a zero row vector of length $q-1$, and the zero in the lower left corner represents a zero column vector of length $q-1$.

Let $G$ denote the compact group $U(q)$. This shows that the fixed-point group $K$ of $e_1$ is isomorphic to $U(q-1)$ and in the following identified with $U(q-1)$. The mapping $A\mapsto Ae_1$ of $G=U(q)$ onto $\Omega_{2q}$ is constant on the left cosets $\xi=AK$ and hence induces a bijection of $G/K$ onto $\Omega_{2q}$. This bijection is a homeomorphism.
The pair $(G,K)$ is known to be a compact Gelfand pair, cf. \cite{vD},\cite{W}.

The bi-invariant functions on $G$ with respect to $K$ can be identified with functions $f:\overline{\D}\to\C$, where $\D=\{z\in\C \mid |z|<1\}$ is the open unit disc and $\overline{\D}$
its closure. To see this notice that the mapping $A\mapsto Ae_1\cdot e_1$ of $G$ onto $\overline{\D}$ is constant on the double cosets and if $Ae_1\cdot e_1=Be_1\cdot e_1$ for $A,B\in G$, then they belong to the same double coset. Therefore the space of double cosets $K\backslash G/K$ is homeomorphic to $\overline{\D}$.

 The bi-invariant functions depend only on the upper left corner $a_{11}$ of $A\in U(q)$.

The image measure of Haar measure $\omega_G$ on $G=U(q)$ under the mapping $A\mapsto Ae_1$ of $G$ onto $\Omega_{2q}$ is the normalized surface measure $\omega_{2q-1}/\sigma_{2q-1}$ on $\Omega_{2q}$. The image measure of $\omega_{2q-1}/\sigma_{2q-1}$ under the mapping $\xi\mapsto \xi\cdot e_1$ of $\Omega_{2q}$ onto $\overline{\D}$ is the probability measure on $\overline{\D}$ given in polar coordinates $z=re^{i\varphi}, 0\le r\le 1,0\le\varphi<2\pi$ as 
$$
\frac{q-1}{\pi} r(1-r^2)^{q-2}dr d\varphi,
$$
cf. formula (2.18) of \cite{K2}.

The spherical functions are precisely the functions $R^{q-2}_{m,n}(z)$ given in \cite{K2},\cite{M:P}, and as bi-invariant functions on $U(q)$ they are  positive definite. 
The functions    $R^{q-2}_{m,n}(z)$ belong to the class of disc polynomials given in \cite{K2} for $\a>-1$ as
\begin{eqnarray*}\label{eq:discpol}
R^\a_{m,n}(r e^{i\varphi})=r^{|m-n|} e^{i(m-n)\varphi}R^{(\a,|m-n|)}_{\min(m,n)}(2r^2-1),\quad 0\le r \le 1,\;0\le \varphi<2\pi
\end{eqnarray*}
and 
$$
R^{(\a,\b)}_k(x)=P^{(\a,\b)}_k(x)/P^{(\a,\b)}_k(1),\quad \a,\b>-1,\;k\in \N_0
$$
are normalized Jacobi polynomials, cf. \cite{A:A:R}.

See \cite{Wu} for other expressions for the disc polynomials.
 
The dual space $Z$ of the compact Gelfand pair $(U(q),U(q-1))$ can be identified with $\N_0^2$ and we have
$$  
\d(R^{q-2}_{m,n}(z))=N(q;m,n)=\frac{m+n+q-1}{q-1}\binom{m+q-2}{q-2}\binom{n+q-2}{q-2}.
$$
For $f\in L^1_K(G)^\sharp$ considered as a function on $\overline{\D}$ we have
$$
\widehat f(m,n)=\frac{q-1}{\pi}\int_0^1\int_0^{2\pi} f(re^{i\varphi})\overline{R^{q-2}_{m,n}(re^{i\varphi})}
r(1-r^2)^{q-2}d\varphi dr.
$$

The space $\mathcal P_K^\sharp(G)$ can be identified with the functions characterized in
Theorem 4.2 in \cite{M:P}, and for an arbitrary locally compact group $L$ the space $\mathcal P_K^\sharp(G,L)$
can be characterized as follows by Theorem~\ref{thm:main}:

\begin{thm}\label{thm:complexsphere} Let $q\in\N, q\ge 2$, let  $L$ be a locally compact group  and let $f:\overline{\D}\times L\to \mathbb C$ be a continuous function. Then
$A\mapsto f(Ae_1\cdot e_1,u)$ belongs to $\mathcal P^\sharp_{U(q-1)}(U(q),L)$ if and only if there exists a double sequence of functions $(\varphi_{m,n}^{q-2})_{m,n\ge 0}$ from $\mathcal P(L) $ with 
$$
\sum_{m,n\ge 0} \varphi_{m,n}^{q-2}(e_L)<\infty
$$ 
such that
\begin{equation}\label{eq:expandcp}
f(z,u)=\sum_{m,n=0}^\infty \varphi_{m,n}^{q-2}(u) R^{q-2}_{m,n}(z),\quad z\in\overline{\D},\;u\in L.
\end{equation}
The above expansion is uniformly convergent on $\overline{\D}\times L$, and we have
\begin{equation}\label{eq:coefcp}
\varphi_{m,n}^{q-2}(u)=N(q;m,n)\frac{q-1}{\pi}\int_{0}^1\int_0^{2\pi}  f(re^{i\varphi},u)\overline{R^{q-2}_{m,n}(re^{i\varphi})}r(1-r^2)^{q-2}\,dr \,d\varphi.
\end{equation}
\end{thm}

In  the above one could  have considered the compact Gelfand pair $(SU(q),SU(q-1))$ when
$q\ge 3$. For $q=2$ this is not a Gelfand pair since $SU(2)$ is a non-abelian group and
  $SU(1)$ is trivial. Furthermore,
$$
(\xi,\eta)\mapsto \begin{pmatrix}
\xi & -\overline{\eta} \\
\eta & \overline{\xi}
\end{pmatrix} 
$$
is a homeomorphism of $\Omega_4$ onto $SU(2)$, and therefore
$\Omega_4$ can be given a group structure so it is isomorphic with $SU(2)$.

\section{Products of Gelfand pairs}

Let $(G_1,K_1)$ and $(G_2,K_2)$ be two Gelfand pairs with dual  spaces $Z_1$ and $Z_2$. Then
$(G_1\times G_2,K_1\times K_2)$ is a Gelfand pair and the dual space $Z$ can be identified with the product space $Z_1\times Z_2$.

In fact, if $\varphi_1\in Z_1, \varphi_2\in Z_2$, then $\varphi_1\otimes\varphi_2:G_1\times G_2\to \C$ defined by $\varphi_1\otimes\varphi_2(x_1,x_2)=\varphi_1(x_1)\varphi_2(x_2)$ is easily seen to belong to $Z$. Furthermore, it is not difficult to see that $(\varphi_1,\varphi_2)\mapsto \varphi_1\otimes\varphi_2$ is a homeomorphism of $Z_1\times Z_2$ onto $Z$.

If we consider the two compact Gelfand pairs $(O(d+1),O(d))$ and  $(O(d'+1),O(d'))$ and apply 
Theorem~\ref{thm:compBG} (ii) to their product, we get Theorem 2.9 of \cite{G:M:P}. Another proof of this theorem was given in \cite[Theorem 6.1]{B:P}.

If we consider the two compact Gelfand pairs $(U(q),U(q-1))$ and $(U(p),U(p-1))$ with $q,p\ge 2$ and apply Theorem~\ref{thm:compBG} (ii) to their product, we get the following result:

\begin{thm}\label{thm:prodcs} Let $f:\overline{\D}\times\overline{\D}\to \C$ be a continuous function and define $F:U(q)\times U(p)\to\C$ by 
$$
F(A,B)=f(Ae_1\cdot e_1,Be_1\cdot e_1),\quad A\in U(q), B\in U(p).
$$
Then $F\in\mathcal P^\sharp_{U(q-1)\times U(p-1)}(U(q)\times U(p))$ if and only if there exists a multi-sequence $c:\N_0^4\to [0,\infty[$ with $\sum c(m,n,k,l)<\infty$ such that
\begin{equation}\label{eq:BGprodcs}
f(z,w)=\sum_{(m,n,k,l)\in\N_0^4} c(m,n,k,l)R^{q-2}_{m,n}(z)R^{p-2}_{k,l}(w),\quad z,w\in\overline{\D}.
\end{equation}
The series \eqref{eq:BGprodcs} converges uniformly on $\overline{\D}^2$.
\end{thm}

\section{An extension from $\mathcal P(L)$ to $\mathcal P(X^2)$}

In a recent paper \cite{G:M} Guella and Menegatto showed how the expansions of functions in
$\mathcal P(\S^d,L)$ with coefficient functions from $\mathcal P(L)$ can be extended to expansions with positive definite kernels on an arbitrary set as coefficient functions. We shall here show how this can be modified to the present framework, where the sphere is replaced by an arbitrary homogeneous space $G/K$ associated with a compact Gelfand pair $(G,K)$.

Let $X$ denote an arbitrary non-empty set and let $\mathcal P(X^2)$ denote the set of kernels
$B:X^2\to\C$ which are positive definite. 

To a function $f:G\times X^2\to \C$ we consider the kernel on
$G\times X$ given by
\begin{equation}\label{eq:kern}
 ((x,u),(y,v))\mapsto f(x^{-1}y,u,v),\quad x,y\in G,\;u,v\in X.
\end{equation}

\begin{defn}\label{thm:ext}
 By $\mathcal P^\sharp_K(G,X^2)$ we shall denote the set of functions $f:G\times X^2\to \C$  satisfying
\begin{enumerate}
\item $f$ is bi-invariant with respect to $K$ in the first variable 
\item $f(\cdot,u,v)$ is continuous on $G$ for each $(u,v)\in X^2$
\item The kernel \eqref{eq:kern} is positive definite.
\end{enumerate}
\end{defn}

Theorem 2.3 in \cite{G:M} can be extended in the following way:

\begin{thm}\label{thm:GM} Let $(G,K)$ be a compact Gelfand pair.
For a function $f:G\times X^2\to \C$  the following are equivalent:
\begin{enumerate}
\item[(i)] $f\in \mathcal P^\sharp_K(G,X^2)$
\item[(ii)] $f$ has a series representation of the form
\begin{equation}\label{eq:expandgen}
f(x,u,v)=\sum_{\varphi\in Z} B(\varphi)(u,v)\varphi(x),\quad x\in G, \;(u,v)\in X^2,
\end{equation}
where $B(\varphi)\in\mathcal P(X^2)$ and $\sum_{\varphi\in Z} B(\varphi)(u,u)<\infty$ for each $u\in X$.
\end{enumerate}
For $f\in \mathcal P^\sharp_K(G,X^2)$ we have
\begin{equation}\label{eq:coef}
B(\varphi)(u,v)=\d(\varphi)\int_G f(x,u,v)\overline{\varphi(x)}\,{\rm d}\omega_G(x),
 \quad (u,v)\in X^2.
\end{equation}
\end{thm}

\begin{proof} The proof is a simple modification of the proof of Theorem 2.3 in \cite{G:M}, so we shall be very brief.

"(i)$\implies $(ii)."

Given $f\in \mathcal P^\sharp_K(G,X^2)$ and $u_1,\ldots,u_n\in X$, $c_1,\ldots,c_n\in\C$ we define 
\begin{equation}\label{eq:comp} 
F(x)=\sum_{j,k=1}^n f(x,u_j,u_k)c_j\overline{c_k}
\end{equation}
and observe that $F\in\mathcal P^\sharp_K(G)$, cf. Lemma 2.2 in \cite{G:M}. By Theorem~\ref{thm:compBG} (ii) we have the series representation
 \begin{equation}\label{eq:compsum} 
F(x)=\sum_{\varphi\in Z} B(\varphi,F)\varphi(x),\quad x\in G,
\end{equation}
where
$$
B(\varphi,F)=\delta(\varphi)\int_G F(x)\overline{\varphi(x)}\,d\omega_G(x) \ge 0
$$
satisfy $\sum_{\varphi \in Z}B(\varphi,F)<\infty$. In particular for $n=2,u_1=u,u_2=v\in X$ and 
$(c_1,c_2)=(1,1),(1,-1), (1,i)$ we get that
\begin{eqnarray*}
F_1(x)&=&f(x,u,u)+f(x,v,v)+f(x,u,v)+f(x,v,u)\\
F_2(x)&=&f(x,u,u)+f(x,v,v)-f(x,u,v)-f(x,v,u)\\
F_3(x)&=&f(x,u,u)+f(x,v,v)-i f(x,u,v)+i f(x,v,u)
\end{eqnarray*}
have convergent expansions like \eqref{eq:compsum}. Since  $f(x,u,v)$ is a linear combination of $F_1,F_2,F_3$, there exists a uniquely determined kernel $B(\varphi):X^2\to \C$ such that
$$
f(x,u,v)=\sum_{\varphi\in Z} B(\varphi)(u,v)\varphi(x),\quad x\in G, u,v\in X,
$$
and the series is absolutely convergent. As in \cite{G:M} we see that $B(\varphi)\in\mathcal P(X^2)$.

"(ii)$\implies $(i)."

If $B(\varphi)\in\mathcal P(X^2)$ satisfies $\sum_{\varphi\in Z}B(\varphi)(u,u)<\infty$ for each $u\in X$, then $\sum_{\varphi\in Z}|B(\varphi)(u,v)|<\infty$ for all $u,v\in X$ because
$$
|B(\varphi)(u,v)|\le (B(\varphi)(u,u)B(\varphi)(v,v))^{1/2}\le (1/2)(B(\varphi)(u,u)+B(\varphi)(v,v)).
$$
It is now easy to see that $f(x,u,v)$ defined by \eqref{eq:expandgen} satisfies (i).
\end{proof}

Theorem 2.7 from \cite{G:M} can also be generalized to the present framework:

\begin{rem}{\rm Suppose in the setting of Theorem~\ref{thm:GM} that  there exists a function $b:Z\to [0,\infty[$ with $\sum_{\varphi\in Z} b(\varphi)<\infty$ such that 
$$
B(\varphi)(u,u)\le b(\varphi),\quad u\in X,
$$
then the expansion \eqref{eq:expandgen} is uniformly convergent for $x\in G,u,v\in X$.
}
\end{rem}

\noindent
Christian Berg\\
Department of Mathematical Sciences, University of Copenhagen\\
Universitetsparken 5, DK-2100, Denmark\\
e-mail: {\tt{berg@math.ku.dk}}

\vspace{0.4cm}
\noindent
Ana P. Peron\\
Departamento de Matem{\'a}tica, ICMC-USP-S{\~a}o  Carlos\\
Caixa Postal 668, 13560-970 S{\~a}o Carlos SP, Brazil\\
e-mail: {\tt{apperon@icmc.usp.br}} 

\vspace{0.4cm}
\noindent
Emilio Porcu \\
Department of Mathematics, Universidad T{\'e}cnica Federico Santa Maria\\
Avenida Espa{\~n}a 1680, Valpara{\'\i}so, 2390123, Chile\\
e-mail: {\tt{emilio.porcu@usm.cl}}


\begin{thebibliography}{xxx}
\bibitem{A:A:R} G.~E.~Andrews, R.~Askey and R.~Roy, Special Functions. Cambridge University Press 1999. 

\bibitem{B} C.~Berg, Dirichlet forms on symmetric spaces, Ann. Institut Fourier, Grenoble {\bf 23} No. 1 (1973), 135--156.

\bibitem{B:C:R} C.~Berg, J.~P.~R.~Christensen and  P.~Ressel,  Harmonic
analysis on semigroups. Theory of positive definite and related
functions. Graduate Texts in Mathematics vol. {\bf 100}.
Springer-Verlag, Berlin-Heidelberg-New York, 1984.

\bibitem{B:P} C.~Berg and E.~Porcu, From Schoenbeg Coefficients to Schoenberg Functions,
Constructive Approximation (to appear).

\bibitem{D} J.~Dixmier, Les $C^*$-alg{\`e}bres et leurs repr{\'e}sentations. Gauthier-Villars, Paris, 1964.

\bibitem{God} R.~Godement, Introduction aux travaux de A. Selberg, S\'eminaire Bourbaki {\bf 9}
(1956-57), expos\'e 144.

\bibitem{G:M} J.C. Guella, V. A. Menegatto, From Schoenberg coefficients to Schoenberg functions: a unifying framework. Manuscript, October 25, 2016.

 \bibitem{G:M:P} J.~C.~Guella, V.~A.~Menegatto and A.~P.~Peron, An extension of a theorem of Schoenberg to products of spheres.  Banach J. Math. Anal. 10 (2016), no. 4, 671--85.  

\bibitem{K2} T.~H.~Koornwinder, The addition formula for Jacobi polynomials II. The Laplace type integral representation and the product formula. Math. Centrum Amsterdam, Report TW133, 1972.

\bibitem{M:P} V.~A.~Menegatto and A.~P.~Peron, Positive Definite Kernels on Complex Spheres.
J. Math. Anal. Appl. {\bf 254} (2001), 219--232.

\bibitem{Mu} J.~R.~Munkres, Topology, Second Edition. Prentice Hall, Inc. New Jersey, 2000.

\bibitem{M} C.~M\"uller, Spherical Harmonics. Lecture Notes in Mathematics. Springer-Verlag.
Berlin, Heidelberg, New York, 1966.

\bibitem{R1} W.~Rudin, Real and Complex Analysis. McGraw-Hill Book Company, Singapore,  1986. 

\bibitem{sasvari} Z.~Sasv{\'a}ri, Positive Definite and Definitizable Functions.  Akademie Verlag, Berlin, 1994.

\bibitem{S} I.~J.~Schoenberg, Positive definite functions on spheres,  Duke Math. J.  {\bf 9} (1942), 96--108.

\bibitem{vD} G.~van Dijk, Introduction to harmonic analysis and generalized Gelfand pairs. De Gruyter Studies in Mathematics, 36. Walter de Gruyter \& Co., Berlin, 2009. x+223 pp. ISBN: 978-3-11-022019-3. 

\bibitem{V:K} N.~Ja.~Vilenkin and A.~U.~Klimyk, Representations of Lie groups and special functions, Vol 2. Kluwer Academic Publishers, Dordrechts, The Netherlands, 1993.

\bibitem{W} J.~A.~Wolf, Harmonic Analysis on Commutative spaces, Mathematical Surveys and Monographs Volume 142, American Mathematical Society, 2007.

\bibitem{Wu} A.~W{\"u}nsche, Generalized Zernike or disc polynomials. J. Comput. Appl. Math. 
{\bf 174} (2005), 135--163.

\end{thebibliography}
\end{document}